\documentclass[a4paper,10pt]{article}

\usepackage[utf8]{inputenc}
\usepackage{amssymb,amsmath}
\usepackage{thm-restate}
\usepackage{amsthm}
\newtheorem{theorem}{Theorem}
\newtheorem{lemma}[theorem]{Lemma}
\newtheorem{question}[theorem]{Question}

\newcommand{\alphabet}{\mathcal{A}}
\newcommand{\minsuff}{\Lambda}

\begin{document}
\author{Matthieu Rosenfeld\thanks{Supported by the ANR project CoCoGro (ANR-16-CE40-0005).}\\
\small \it University of Montpellier, LIRMM}
\title{Avoiding squares over words with lists of size three amongst four symbols}

% \title{Avoiding squares over words with lists of size three amongst four symbols
% }
% 
% \author{Matthieu Rosenfeld\thanks{Supported by the ANR project CoCoGro (ANR-16-CE40-0005).}}
% 
% \institute{University of Montpellier, LIRMM.}

\maketitle

\begin{abstract}
In 2007, Grytczuk conjecture that for any sequence $(\ell_i)_{i\ge1}$ of alphabets of size $3$ there exists a square-free infinite word $w$ such that for all $i$, the $i$-th letter of $w$ belongs to $\ell_i$. 
The result of Thue of 1906 implies that there is an infinite square-free word if all the $\ell_i$ are identical. On the other, hand Grytczuk, Przybyło and Zhu showed in 2011 that it also holds if the $\ell_i$ are of size $4$ instead of $3$.

In this article, we first show that if the lists are of size $4$, the number of square-free words is at least $2.45^n$ (the previous similar bound was $2^n$). 
We then show our main result: we can construct such a square-free word if the lists are subsets of size $3$ of the same alphabet of size $4$. Our proof also implies that there are at least $1.25^n$ square-free words of length $n$ for any such list assignment.
This proof relies on the existence of a set of coefficients verified with a computer. We suspect that the full conjecture could be resolved by this method with a much more powerful computer (but we might need to wait a few decades for such a computer to be available). 
\end{abstract}

\section{Introduction}
A \emph{square} is a word of the form $uu$ where $u$ is a non-empty word.
We say that a word is \emph{square-free} (or avoids squares) if none of its factors is a square.
For instance, $hotshots$ is a square while $minimize$ is square-free.
In 1906, Thue showed that there are arbitrarily long ternary words avoiding squares \cite{Thue06,Thue1}.
This result is often regarded as the starting point of combinatorics on words, and the generalizations of this particular question received a lot of attention.

\emph{Nonrepetitive colorings} of graphs were introduced by Alon~\textit{et~al.} \cite{alongraph}. 
A coloring of the vertices (or of the edges) of a graph is said to be \emph{nonrepetitive} if there is no path of the graph whose color sequence is a square. The \emph{nonrepetitive chromatic number} (resp. \emph{nonrepetitive chromatic index}) of a graph is the minimal number of colors in a nonrepetitive coloring of the vertices (resp. the edges) of the graph.
Alon et al. showed that nonrepetitive chromatic index is as most $O(\Delta^2)$ where $\Delta$ is the maximum degree of $G$ \cite{alongraph}.
Different authors successively improved the upper bounds on the nonrepetitive chromatic number and the nonrepetitive chromatic index and the best known bound for the nonrepetitive chromatic number is also in $O(\Delta^2)$ \cite{Dujmovic2016,MontassierEntropie,HARANT2012374,rosenfeldCounting}. 
Non-repetitive colorings have since been studied in many other contexts (see for instance \cite{woodSurvey} for a recent survey on this topic). 

Most results regarding non-repetitive colorings of graphs of bounded maximal degree are based on the Lov\'asz Local Lemma, entropy-compression, or related methods and they naturally hold in the stronger setting of list coloring. 
In this setting, each vertex is assigned a list of colors and the colorings of the graph must assign to each vertex a color from its list.
The \emph{nonrepetitive list chromatic number} is the smallest integer $n$ such that the graph is nonrepetitively colorable as soon as all the lists contain at least $n$ colors.
This notion was studied in relation to many notions of colorings and contrary to the first intuition it is often not the case that the worst possible list assignment is the one that gives the same list to every vertex. For instance, every planar graph has nonrepetitive chromatic number at most 768, but can have an arbitrarily large nonrepetitive list chromatic number  \cite{planargraphs}. However, the best known bound on the nonrepetitive chromatic number in terms of the maximal degree also holds for the nonrepetititive list chromatic number \cite{woodSurvey}. It is unknown whether the optimal bounds in terms of the maximal degree also holds for the nonrepetititive list chromatic number are indeed identical or not.
The simplest graph for which the question is non-trivial is the path. 
The result of Thue implies that the nonrepetitive chromatic number of any path over at least four vertices is $3$ \cite{Thue06}. There are various simple proofs that the nonrepetitive list chromatic number of any path is at most 4 (see \cite{GrytczukLLLlist} for a proof based on the Lefthanded Local Lemma,
\cite{GrytczukGame} for a proof based on entropy compression and \cite{rosenfeldCounting} for a proof based on a simple counting argument). 
It was first conjectured by Grytczuk \cite{Grytczukpathquestion} that the nonrepetitive list chromatic number of any path is in fact at most 3.
This conjecture has been mentioned many times, but not much progress has been made in the direction of proving or disproving it (see for instance \cite{questionpathpathwidth,Grytczukpathquestion,Grytczukpathquestion2,GrytczukGame,rosenfeldCounting,Skrabuvlakova2015TheTC,woodSurvey,deuxplusepsilon}
%questionpath3
for some of the occurences of this problem).

The question can be reformulated in terms of combinatorics on words.
\begin{question}\label{mainquest}
Let $\alphabet$ be an infinite alphabet.
Is it true that for any sequence $(\ell_i)_{i\ge1}$ of subsets of $\alphabet$ of size $3$ there exists an infinite square-free word $w$ such that for all $i$, the $i$-th letter of $w$ belongs to $\ell_i$?
\end{question}
As already mentioned the answer is positive if $3$ is replaced by $4$. It was even shown in \cite{rosenfeldCounting} that there are at least $2^n$ such words of length $n$ for any list assignment.
We first show in this article, that there are at least $2.45^n$ such words.
We then show our main result: if $\alphabet$ is of size $4$ then there is such a word. 
Our approach is similar to the idea of \cite{rosenfeldCounting}, that is, we show some strong bounds on the number of such words using an inductive argument.
Our approach also relies on ideas of an approach of Shur to bounds the number of words power-free languages \cite{shurgrowthrate} (this improves an older technique introduced by Kolpakov \cite{Kolpakov2007}). This idea is that instead of directly counting the words we associate a weight to each word and we count the total weight of the set of valid words. In the approach of Kolpakov, he had to deal with three different kinds of squares (short squares, long squares, and mid-length squares) and we borrow a trick from Shur that allows dealing with only two kinds of squares (\textit{i.e.}, the mid-length squares do not need to be dealt with separately). In fact, the main difference between our proof and the approach of Shur is that eigenvectors and eigenvalues are replaced by a more complicated notion (which does not seems to be well studied or even named in the litterature).
For our main result, we use a computer to verify the existence of a set of weights with the right property. This approach could probably be applied if $\alphabet$ is of size larger than $4$ (and even for infinite $\alphabet$), but the computational power required is larger than what modern computers offer. In particular, we suspect that this approach would work if $\alphabet$ was of size $21$ and that would imply the conjecture.

The article is organized as follows. We first provide some notations and definitions in Section \ref{secdef}. 
In Section \ref{sec4letters}, we show that for any list assignment with lists of size $4$, there are at least $2.45^n$ square-free words of size $n$. We use this proof to introduce our technique. In section \ref{secmainred}, we use this technique to show our main result. In Section \ref{proofLemmaComputer}, we detail how we use the computer to verify the existence of the set of weights needed for our result. 
We conclude this article, in Section \ref{secconcl}, by explaining why we suspect that the same approach could solve Question \ref{mainquest}.

\section{Definitions and notations}\label{secdef}
We use the standard definitions and notations of combinatorics on words introduced in Chapter 1 of \cite{lothaire}.

For the sake of notations, all our alphabets are sets of integers. For any fixed alphabet $\mathcal{A}$, the \emph{set of extensions of a word $w$} is the set of words $\{wa: a\in \alphabet\}$. That is, a word $u$ is an extension of another word $w$ is u can be obtained by the concatenation of a letter at the end of $w$. For any words $w=w_1w_2\ldots w_n$ and any $i,j$  such that $1\le i\le j\le n$, we write $w[i,j]= w_iw_{i+1}\ldots w_j$.

A \emph{square} is a word of the form $uu$ with $u$ a non-empty word. The \emph{period} of the square $uu$ is $u$ and by abuse of notation we also call the length $|u|$ of $u$ the period of $uu$. A word is \emph{square-free} (or \emph{avoids squares}) if none of its factors is a square.

A \emph{list assignement} is a sequence $(\ell_i)_{i\ge1}$ of subsets of the integers.
A \emph{$k$-list assignement} is a list assignement $(\ell_i)_{i\ge1}$
such that each list is of size 4, \textit{i.e.} for all $i\ge1$, $|\ell_i|=4$.
We say that a word $w=w_1\ldots w_n$ \emph{respects} a list assignement $(\ell_i)_{i\ge1}$ if for all $i\in\{1,\ldots,n\}$, $w_i\in \ell_i$.

\section{The number of square-free words over 4 letters}\label{sec4letters}

We say that a word of length at least $3$ is \emph{perfect} if its suffix of length $3$ contains $3$ distinct letters (it is $012$ up to a permutation of the alphabet). A word is \emph{nice} if its suffix of length 3 is $010$ up to a permutation of the alphabet. A square-free word is either nice or perfect. The idea is that since nice words are intuitively easier to extend than nice words it is better to count them separately. However, since it is complicated to count them separately, we will count them together by weighting them. 

For any set of words $S$, $\widehat{S}$ is  the quantity obtained by summing the number of perfect words in $S$ together with $\sqrt3-1$ times the number of nice words in $S$. For instance, if $S=\{021,010,101,212\}$ then $\widehat{S}=3\sqrt3-2$.

\begin{lemma}\label{fourletterslemma}
Let $\ell=(\ell_i)_{i\ge1}$ be a $4$-list assignement and let $T_n$ be the number of square-free words of length $n$ that respect $\ell$.
Let $\beta>1$ be a real such that $1+\sqrt3 -  \frac{1}{\beta(\beta-1)}\ge \beta$.
Then, for any $n\ge3$,
$$\widehat{T_{n+1}}\ge\beta\widehat{T_{n}}\,.$$
\end{lemma}
\begin{proof}
  We proceed by induction on $n$. 
  Suppose that for every $i<n$, $\widehat{T_{i+1}}\ge\beta\widehat{T_{i}}$. Then, for all $i$,
  \begin{equation}\label{IHPS4letters}
\widehat{T_{n}}\ge \beta^i\widehat{T_{n-i}}\,.
  \end{equation} We need to show $\widehat{T_{n+1}}\ge\beta\widehat{T_{n}}\,.$

A word of length $n+1$ is \emph{good} if
\begin{itemize}
  \item it respects $\ell$,
  \item its prefix of length $n$ is in $T_{n}$,
  \item and it contains no square of period $1$ or $2$. 
\end{itemize} 
A word is \emph{wrong}, if it
is good, but not square-free (i.e., if one of its suffixes is a square of period longer than $2$).
Let $G$ be the set of good words and $F$ be the set of wrong words.
Then $T_{n+1}=G\setminus F$ and
\begin{equation}\label{eqSGFfourletters}
  \widehat{T_{n+1}}\ge \widehat{G}-\widehat{F}\,.
\end{equation}
We will now lower bound $\widehat{G}$ and then upper bound $F$ to reach our result.

Let $u$ be a perfect word from $T_{n}$ and let $abc$ be the factor of size $3$ at the end of $u$. Let us count the contributions to $G$ of the extensions of $u$.
Since $u$ is perfect the only reason for a square of period $1$ or $2$ to appear when adding a letter to $u$ is if this letter is $c$, so we need to forbid at most one letter of $\ell_{n+1}$. Moreover, if $b$ belongs to $\ell_{n+1}$ then there are at least 2 ways to extend $u$ into a perfect word of $G$ and one way to extend $u$ to a nice word of $G$. In this case, the contribution of the extensions of $u$ to $\widehat{G}$ is at least $2+(\sqrt3-1)=1+\sqrt3$. If $b$ does not belong to  $\ell_{n+1}$ then there are at least $3$ ways to extend $u$ to a perfect word of $G$ and the contribution of the extensions of $u$ to  $\widehat{G}$ is at least $3$. Thus the contribution to $\widehat{G}$ of the extensions of any perfect word from $T_{n}$ is at least $1+\sqrt3$. So the contribution of the extensions of any perfect word to $\widehat{G}$ is at least $1+\sqrt3$ times as large as its contribution to $\widehat{T_{n}}$.

For any value of the list $\ell_{n+1}$, any nice word from $T_{n}$ can be extended in at least 2 perfect words of $G$.
The contribution to $\widehat{G}$ of the extensions of any nice word from $T_{n}$ is at least $2$. Since the contribution of a nice word to $\widehat{T_n}$ is $\sqrt3-1$, the contribution of the extensions of any nice word to $\widehat{G}$ is at least $\frac{2}{\sqrt3-1}=1+\sqrt3$ times as large as its contribution to $\widehat{T_{n}}$.

Since the contribution of the extensions of any word to $\widehat{G}$ is at least $1+\sqrt3$ times as large as its contribution to $\widehat{T_{n}}$, we deduce

\begin{equation}\label{boundOnGfourletters}
\widehat{G}\ge(1+\sqrt3)\widehat{T_{n}}\,.
\end{equation}

Let us now bound $F$.
For all $i\ge1$, let $F_{i}$ be the set of words from $F$ that end with a square of period $i$.
Clearly, $F= \cup_{i\ge1} F_{i}$ and 
\begin{equation}\label{FFifourletters}
  \widehat{F}\le \sum_{i\ge1} \widehat{F_{i}}\,.
\end{equation}

By definition of $G$ and $F$, $|F_1|=|F_2|=0$.
Let $i\ge3$ and $u\in F_{i}$.
Since $u$ ends with a square of period $i$, the last $i$ letters of $u$ are uniquely determined by its prefix $v$ of size $n+1-i$. Since $v$ is a proper prefix of $u$, $v\in T_{n+1-i}$. And moreover, the last 3 letters of $v$ are identical to the last 3 letters of $u$. 
Thus the contribution of $v$ to $\widehat{T_{n+1-i}}$ is the same as the contribution of $u$ to $\widehat{F_{i}}$. We deduce
$$\widehat{F_{i}}=\widehat{T_{n+1-i}}\,.$$
Using this last equation with \eqref{IHPS4letters} and \eqref{FFifourletters} yields
$$
 \widehat{F}\le \sum_{i\ge 3} \widehat{F_i}
 \le \sum_{i\ge 3} \widehat{T_{n+1-i}}
 \le \sum_{i\ge 3} \widehat{T_n} \beta^{1-i}
 \le  \frac{\widehat{T_n}}{\beta(\beta-1)}\,.$$
We can use this bound and \eqref{boundOnGfourletters} in equation 
\eqref{eqSGF} to finally upper-bound $\widehat{T_{n+1}}$ and we obtain
\begin{equation*}
  \widehat{T_{n+1}}\ge \left(1+\sqrt3 -  \frac{1}{\beta(\beta-1)}\right)\widehat{T_n}\ge \beta\widehat{T_n}
\end{equation*}
where the second inequality is a consequence of the theorem hypothesis. 
\end{proof}

Since $\beta=2.45$ satisfies the condition of Lemma \ref{fourletterslemma}, we deduce the following theorem.
\begin{theorem}\label{fourletterstheorem}
  Fix a $4$-list assignement $(\ell_i)_{i\ge1}$ and let $T_{n}$ be the set of square-free words of length $n$ that respect this list assignement. Then for all $n\ge1$,
  $$|T_{n}|\ge 1.45^n\,.$$
\end{theorem}
\begin{proof}
 It is easy to verify that $T_3$ is not empty and thus $\widehat{T_3}>0$.
 One easily verifies that $\beta =2.45$ satisfies the conditions of Theorem \ref{technicalthm}. Thus for all $n\ge3$, 
 $$\widehat{T_n}\ge \frac{\widehat{T_3}}{2.45^3}\times 2.45^{n}\,.$$
Since the weight of any words is at most $1$, for any $n\ge3$,
 \begin{equation}\label{sizeofTn}
   |T_n|\ge\widehat{T_n}\ge \frac{\widehat{T_3}}{2.45^3}\times 2.45^{n}\,.
 \end{equation}

  For the sake of contradiction suppose that  $|T_{n}|< 2.45^n$ for some $n>1$. Then, there exists an integer $i$ such that $|T_{n}|^i<\frac{\widehat{T_3}}{2.45^3}\times 2.45^{ni}$.
Any factor of a square-free word is square-free, so the sequence $(|T_n|)_{n\ge0}$ is submultiplicative, that is, for all $i,j\ge1$, $|T_{i+j}|\le |T_{i}||T_{j}|$. 
In particular, $|T_{ni}|\le|T_{n}|^i<\frac{\widehat{T_3}}{2.45^3}\times 2.45^{ni}$ which contradicts equation \eqref{sizeofTn}. We deduce that for all $n$, $|T_{n}|\ge 2.45^n$, as desired.
\end{proof}

Let us briefly explain why this value of $1+\sqrt3$ plays a particular role in this proof. If the lists are all identical then any perfect word can be extended in two ways into another perfect word and in one way into a nice word, while any nice word can only be extended in two ways into a perfect word. The matrix of the corresponding automaton is
$$\begin{pmatrix}
    2&1\\
    2&0
  \end{pmatrix}\,.$$
  The dominant eigenvalue of this matrix is $1+\sqrt3$ and one possible eigenvector is $\begin{pmatrix}1\\\sqrt3 -1\end{pmatrix}$ which also explains the choice of the weights.
  We do not need for our weights to be an eigenvector, but we require a property similar to the statement of Lemma \ref{bycomputer}. It happens to be the case that the vector that has this property is also the eigenvector of this matrix, which corresponds to the intuition that the worst choice of list assignment is the one where all the lists are identical.

\section{Proof of the main result}\label{secmainred}
We fix $\alphabet=\{0,1,2,3\}$.
An \emph{$(\alphabet,3)$-list assignement} is a sequence of subsets of size $3$ of $\alphabet$.
We also fix $p=21$.

A word is \emph{normalized} if it is the smallest of all the words obtained by a permutation of the alphabet.
Let $\minsuff$ be the set of normalized prefixes of minimal squares of period at most $p$.
For any $w$, we let $\minsuff(w)$ be the longest word from $\minsuff$ that is a suffix of $w$ up to a permutation of the alphabet.
For any set of words $S$, and any $w\in \minsuff$, we let $S^{(w)}$ be the set of words from $S$ whose longest prefix that belongs to $\minsuff$ up to a permutation of the alphabet is $w$, that is $S^{(w)}=\{u\in S: \minsuff(u)=w\}$.

We denote the set of words that contain no square of period at most $p$ by $\mathcal{S}_{free}^{\le p}$. We are now ready to state the following lemma.

\begin{restatable}{lemma}{restcomp}
\label{bycomputer}
There exist coefficients $(C_w)_{w\in\minsuff}$ such that $C_\varepsilon>0$ and for all $v\in \minsuff$,
\begin{equation}\label{eqpassage}
\alpha C_v\le \min_{\substack{l\subseteq\alphabet\\|l|=3}} \sum_{\substack{a\in l\\ va\in \mathcal{S}_{free}^{\le p}}} C_{\minsuff(va)}
\end{equation}
where $\alpha = 13948 / 10721  \approx  1.301$.
\end{restatable}

The proof of this lemma relies on a computer verification that we delay to section \ref{proofLemmaComputer}. For the rest of this section let us fix coefficients $(C_w)_{w\in\minsuff}$ and $\alpha$ that respect the conditions of Lemma \ref{bycomputer}.
For each set $S$ of words, we let 
$$\widehat{S} =\sum_{w\in \minsuff} C_w |S^{(w)}|\,.$$
Whenever we mention \emph{the weight of a word $w$},
we mean $C_{\minsuff(w)}$.

The main idea of the proof of Theorem \ref{technicalthm} is essentially the same as in Lemma \ref{fourletterslemma}, but we have a few more technicities to handle along the way.
We are going to count inductively the total weight of the square-free words of size $n$ that respect a fixed $3$-list assignment.
Intuitively, the set $\minsuff$ plays the same role as $\{010,012\}$ and $\alpha$ plays the same role as $1+\sqrt3$.
We are now ready to state our main Theorem.
\begin{theorem}\label{technicalthm}
Let $\ell=(\ell_i)_{i\ge1}$ be a $(\Sigma,3)$-list assignement  and for all $n\ge0$, let $S_{n}$ be the set of square-free words of length $n$ that respect $\ell$.
Let $\beta>1$ be a real number such that
$$\alpha -  \frac{ \beta^{1-p}}{\beta-1}\ge \beta\,.$$
  Then for all  $n\ge0$, 
  $$\widehat{S_{n+1}}\ge \beta\widehat{S_n}\,.$$
\end{theorem}
\begin{proof}
  We proceed by induction on $n$. Let $n$ be an integer such that the lemma holds for any integer smaller than $n$ and let us show that 
  $\widehat{S_{n+1}}\ge \beta\widehat{S_n}$.
  
  By induction hypothesis, for all $i$,
  \begin{equation}\label{IHPS}
\widehat{S_{n}}\ge \beta^i\widehat{S_{n-i}}\,.
  \end{equation}

A word of length $n+1$ is \emph{good} if
\begin{itemize}
  \item it respects $\ell$,
  \item its prefix of length $n$ is in $S_{n}$,
  \item and it contains no square of period at most $p$.
\end{itemize} A word is \emph{wrong}, if it
is good, but not square-free (i.e., if one of its suffixes is a square of period longer than $p$).
We let $G$ be the set of good words and $F$ be the set of wrong words.
Then $S_{n+1}=G\setminus F$ and

\begin{equation}\label{eqSGF}
  \widehat{S_{n+1}}\ge \widehat{G}-\widehat{F}\,.
\end{equation}

Let us first lower-bound $\widehat{G}=\sum_{w\in \minsuff} |G^{(w)}|C_w$.

By definition, $\minsuff(v)$ is the longest suffix of $v$ that is prefix of a square of period of length at most $p$ (up to permutation of the alphabet). This implies that for any square-free word $v$ and for any word $u$, $vu\in\mathcal{S}_{free}^{\le p}$ if and only if $\minsuff(v)u\in\mathcal{S}_{free}^{\le p}$. For the same reason, for any square-free word $v$ and letter $a$, $\minsuff(va)=\minsuff(\minsuff(v)a)$.
We then deduce that the contribution of the extentions of any word $v\in S_n$ to $\widehat{G}$ is 
$$\sum\limits_{\substack{a\in  \ell_{n+1}\\ va\in \mathcal{S}_{free}^{\le p}}} C_{\minsuff(va)}=
\sum\limits_{\substack{a\in  \ell_{n+1}\\ \minsuff(v)a\in \mathcal{S}_{free}^{\le p}}} C_{\minsuff(\minsuff(v)a)}
\ge\min_{\substack{l\subseteq\alphabet\\|l|=3}} \sum_{\substack{a\in l\\ \minsuff(v)a\in \mathcal{S}_{free}^{\le p}}} C_{\minsuff(\minsuff(v)a)}\,.$$
By Lemma \ref{bycomputer}, we deduce that the contribution of the extentions of any word $v\in S_n$ to $\widehat{G}$ is at least $\alpha C_{\minsuff(v)}$.
We sum the contributions over $S_n$ 
\begin{equation}\label{boundOnG}
\widehat{G}\ge\sum_{v\in S_n}\alpha C_{\minsuff(v)}
=\sum_{u\in \minsuff}\alpha C_u |S_n^{(u)}|=\alpha \widehat{S_{n}}\,.
\end{equation}

Let us now bound $F$.
For all $i$, let $F_{i}$ be the set of words from $F$ that end with a square of period $i$.
Then $F= \cup_{i\ge1} F_{i}$ and 
\begin{equation}\label{FFi}
  \widehat{F}\le \sum_{i\ge1} \widehat{F_{i}}\,.
\end{equation}

Let us now upper-bound the $\widehat{F_{i}}$ separately depending on $i$.

\paragraph{Case $i \le p$ :} 
By definition of $G$ and $F$, any word from $F$ avoids squares of period at most $p$. Hence, $i \le p$ implies $|F_i|=0$ and $\widehat{F_{i}}=0$.

\paragraph{Case $i > p$ :} 
Let $w\in \minsuff$ and $u\in F_i^{(w)}$.
For the sake of contradiction, suppose that $|w|> i$. Let $v$ be the period of $w$. 
Then the factor of length $i-|v|$ that ends a position $n+1-|v|$ is identical to the last $i-|v|$ letters of the word, that is
$$u[n+2-i,n+1-|v|]=u[n+2-i+|v| , n+1]\,.$$
These two factors are well defined since $|v|\le p<i$.
 Moreover $u\in F_i^{(w)}$, and  the factor of length $i-|v|$ that ends at position  $n+1-i$ is identical to the last $i-|v|$ letters of the word, that is
 $$u[n+2-2i+|v|,n+1-i]=u[n+2-i+|v| , n+1]\,.$$
 The resulting equality $$u[n+2-i,n+1-|v|]=u[n+2-2i+|v|,n+1-i]$$ implies that there is a square of period $i-|v|\ge1$ in $u$. 
Since $w$ is square free, we have $2|v|>|w|\ge i$ and $i-|v|<|v|\le p$. So the square is of period at most $p$ which is a contradiction. Hence if $F_i^{(w)}$ is non-empty, then $|w|\le i$. 

The suffix of length $2i$ of any word from $F_i^{(w)}$ is a square, so the last $i$ letters of any word of $F_i^{(w)}$ are uniquely determined by the remaining prefix, and this prefix belongs to $S_{n+1-i}^{(w)}$ (since $|w|\le i$). 
Hence,
$$\widehat{F_{i}}\le \widehat{S_{n+1-i}}\le \widehat{S_n} \beta^{1-i}$$
where the second inequality comes from \eqref{IHPS}.
This bound with \eqref{FFi} yields
$$
 \widehat{F}\le \sum_{i\ge p+1} \widehat{S_n} \beta^{1-i}\le \widehat{S_n} \frac{ \beta^{1-p}}{\beta-1}\,.$$
We use this bound and \eqref{boundOnG} in equation 
\eqref{eqSGF} to finally upper-bound $\widehat{S_{n+1}}$,
\begin{equation*}
  \widehat{S_{n+1}}\ge \left(\alpha -  \frac{ \beta^{1-p}}{\beta-1}\right)\widehat{S_n}\ge \beta\widehat{S_n}
\end{equation*}
where the second inequality is a consequence of the theorem hypothesis.
\end{proof}

Since $\beta=1.25$ satisfies the conditions of Theorem \ref{technicalthm}, we can deduce our main result.
\begin{theorem}\label{mainthm}
  Fix a $(\Sigma,3)$-list assignement $(\ell_i)_{i\ge1}$ and let $S_{n}$ be the set of square-free words of length $n$ that respect this list assignement. Then for all $n\ge1$,
  $$|S_{n}|\ge 1.25^n\,.$$
\end{theorem}
\begin{proof}
 By definition, $\widehat{S_0}= C_{\varepsilon}$.
 One easily verifies that $\beta =1.25$ satisfies the conditions of Theorem \ref{technicalthm}. Thus for all $n$, 
 $$\widehat{S_n}\ge C_{\varepsilon}\times 1.25^n\,.$$
 Hence,
 \begin{equation}\label{mainSnbound}
|S_n|\ge \frac{\widehat{S_n}}{\max\limits_{w\in\minsuff}C_w}\ge \frac{C_{\varepsilon}}{\max\limits_{w\in\minsuff}C_w}\times 1.25^n\,.   
 \end{equation}
For the sake of contradiction suppose that  $|S_{n}|< 1.25^n$ for some $n>1$. By Lemma \ref{bycomputer}, $C_{\varepsilon}>0$ so there exists an integer $i$ such that $|S_{n}|^i<\frac{C_{\varepsilon}}{\max_{w\in\minsuff}C_w}\times1.25^{ni}$.  
Any factor of a square-free word is square-free, so the sequence $(|S_n|)_{n\ge0}$ is submultiplicative, that is, for all $i,j\ge1$, $|S_{i+j}|\le |S_{i}||S_{j}|$. 
In particular, $|S_{ni}|\le|S_{n}|^i<C_{\varepsilon}\times1.25^{ni}$ which contradicts equation \eqref{mainSnbound}. We deduce that for all $n$, $S_{n}\ge 1.25^n$, as desired.
\end{proof}

\section{Proof of Lemma \ref{bycomputer}}\label{proofLemmaComputer}
A classic procedure to compute Perron-Frobenius egeinvector (and eigenvalue) of a matrix
is simply to iterate the matrix over some ``random'' starting vector.
Our coefficients $(C_w)_{w\in\minsuff}$ plays a similar role to the role of an eigenvector and we use the same idea to compute the desired coefficients.
In the case of Perron-Frobenius eigenvector, it is known that under general assumptions this procedure converges toward the desired vector. In our case, we suspect that the procedure is also convergent under rather general assumptions, but we did not try to prove anything in this direction. However, after enough iterations, we find a vector that respects the conditions of our Lemma.
Before discussing the details of this procedure we restate the Lemma.

\restcomp*

It is enough to provide coefficients with the right property. 
Since $|\minsuff|= 854683883$, instead of providing the coefficients, it is more efficient to provide a computer program that computes the set $\minsuff$ and the coefficients with the desired properties\footnote{The C++ implementation can be found in the ancillary file on the arXiv. Running this program took approximately 2 hours of computation and occupied 76.4 Go of RAM.}. This also has the advantage that the same program directly verifies that the coefficients do indeed have the desired property.

The first step is to compute the set $\minsuff$ of prefixes of minimal squares of period at most $p$. This set is stored inside a trie (also called prefix tree). It is efficient (in particular in terms of memory consumption) since the set $\minsuff$ is prefixed closed (\textit{i.e.} $\minsuff$ contains any prefix of any word from $\minsuff$).
Each word of $\minsuff$ is given a unique integer as an identifier.
In a second step, we compute the directed multi-graph $G$ over the vertices $\{0,\ldots, |\minsuff|-1\}$ and such that there is an arc from $i$ to $j$, if $\minsuff(ua)$ is the word corresponding to $j$ where $u$ is the word corresponding to $i$ and $a$ is any letter. More precisely, if $u$ is the word associated with the integer $i$ and $v$ is the word associated with $j$ then the multiplicity of the number of arcs from $i$ to $j$ in our graph is given by $|\{a\in\alphabet : u=\minsuff(va)\}|$. There are not many arcs with multiplicity larger than 1, but because of the normalization, this may happen (for instance, there are three arcs from the vertex of the word $0$ to the vertex of the word $01$, since $01$ is the normalization of $02$ and $03$ and similarly there are four arcs from the empty word to the word $0$). This graph is useful to efficiently compute for any $v\in\minsuff$, the quantity $\min\limits_{\substack{l\subseteq\alphabet\\|l|=3}} \sum\limits_{\substack{a\in l\\ va\in \mathcal{S}_{free}^{\le p}}} C_{\minsuff(va)}$. It is simply the sum of the weights of its out-neighbors minus the weight of the one of largest weight\footnote{A vertex $u$ is said to be \emph{an out-neighbor} of a vertex $v$ if there is an arc from $v$ to $u$.}. 

We considere the procedure that takes coefficients $(C_w)_{w_\in\minsuff}$ as input and produces the coefficients $(C'_w)_{w_\in\minsuff}$ such that for each $v\in\minsuff$
$$C'_v = \min_{\substack{l\subseteq\alphabet\\|l|=3}} \sum_{\substack{a\in l\\ va\in \mathcal{S}_{free}^{\le p}}} C_{\minsuff(va)}\,.$$

For every $w$, we call the quantity $C'_w/C_w$ the \emph{growth associated to $w$}. If we let $\alpha$ be the minimum of the growth over every $w\in\minsuff$ then $\alpha$ and the set of coefficients $C_w$ respect the condition of equation \ref{eqpassage}. Our goal is then simply to find coefficients $(C_w)_{w_\in\minsuff}$ that gives the largest value of $\alpha$.

To find, our coefficient we simply start by setting all the $C_w$ to the same value ($100000$ in our implementation), and then we iterate our procedure. 
Between two iterations we renormalize the coefficients by dividing every coefficient by the same constant to keep the average value at some fixed value ($100000$ in our implementation).
After 50 iterations, we find coefficients $(C_w)_{w_\in\minsuff}$ and $\alpha = 13948 / 10721  \approx  1.301$ such that the conditions of the Lemma are verified.

We suspect that there are good reasons for which this procedure seems to converge toward the optimal. However, it is enough that we verified that after 50 iterations this deterministic procedure produces coefficients $(C_w)_{w_\in\minsuff}$ with the desired property. 

Let us finally mention that, every computation is carried out using integers (or pairs of integers for rational numbers) so that there are no issues of precision. It is important, since with $\alpha > 1.3$ we can take $p=21$ and $\beta=1.25$ in Theorem \ref{mainthm}, but with $\alpha=1.295$ there is no $\beta$ that satisfy the condition of Theorem \ref{technicalthm}.

\section{Conclusion}\label{secconcl}
We showed something slightly stronger than Theorem \ref{mainthm}.
Indeed, our proof of Theorem \ref{technicalthm} also holds if an adversary chooses each list only right before we chose the letter from the corresponding list (instead of fixing the lists from the start we have). Our proof of Theorem \ref{fourletterstheorem} also holds in this stronger setting. We suspect that the answer to Question \ref{mainquest} is positive in this stronger setting. 

We showed that for any choice of lists of size $3$ amongst $\{0,1,2,3\}$ the number of square-free word of length $n$ is at least $1.25^n$. By pushing the computation slightly further, we can replace $1.25$ with $1.28$.
We suspect that the number of square-free words is minimal when all the lists are identical (the growth rate of the number of square-free word of length $n$ is known to be approximatively 1.3017 \cite{shurgrowthrate}). This might even be true in the stronger context where an adversary chose the next list right after we chose the next letter. Note that, it is also really simple to use the same technique to improve the bounds of Theorem \ref{fourletterstheorem} by computing coefficients with the aid of a computer (although looking at suffixes of length $5$ or $7$ instead of $3$ by hand is doable and would already improve the bounds).

Finally, let us conclude by mentioning that we believe that this approach could be used to answer positively Question \ref{mainquest}. The size of the set $\minsuff$ depends on the size of $\alphabet$. However, since we only consider normalized words for $p=21$ the set $\minsuff$ is the same for any $\alphabet$ such that $|\alphabet|\ge21$. Let us provide a crude upper bound on the size of this set.

The number of square-free words of size $n$ that use $k$ (by that we mean that the $k$ letters appear in the word) letters is less $k\times(k-1)^{n-1}$ (we forbid square of period $1$). So the number of normalized words of size $n$ that use $k$ letters is less than $k\times(k-1)^{n-1}/k!=(k-1)^{n-2}/(k-2)!$.
The number of normalized minimal square of period $n$ that use $k$ letter is less than $(k-1)^{n-2}/(k-2)!$ and the number of proper prefixes of length at least $n$ of such words is then at most $n\times (k-1)^{n-2}/(k-2)!$. By summing over $n$ and $k$, the number of prefixes of minimal square-free words of period at most $p$ over $s$ letters is at most
$$\sum_{n=1}^{p}\sum_{k=1}^s n\frac{(k-1)^{n-2}}{(k-2)!}\,.$$
Evaluating this expression with $p=21$ and $s=21$ tells us that ,with $|\alphabet|=21$, the set $\minsuff$ has size at most $3.4\times10^{15}$. This is less than $10^7$ times larger than the set that required $70$Go of RAM to be computed. Our bound being crude we suspect that $\Lambda$ is in fact smaller than that (a more careful computation taking into account squares of period $2$ divides this bounds by approximatively $4$). We might be able to solve this question with this approach in a few decades. It might also be possible to exploit some other symmetries of the problem to reduce the number of words considered.

Let us finally mention that the base idea of the counting technique was recently used in more general context \cite{rosenfeldCounting,wanlessWood}. In particular, Wanless and Wood provided a general result based on a similar idea and applied it to graph colorings, hypergraph colorings and SAT-formula \cite{wanlessWood}. It is not clear, how the more advanced argument used here (\textit{i.e.}, counting the total weights of the solutions instead of simply counting the number of solutions) can be used in a context more general than combinatorics on words or even whether it can be used in a framework similar to the one developed by Wanless and Wood.

\bibliographystyle{splncs03}

\end{document}